\documentclass[12pt,leqno,amscd,amssymb,verbatim, url]{amsart}
\usepackage{amsfonts,amssymb}
\usepackage[dvips]{color}
\usepackage{amsmath,amscd}
\newtheorem{thm}{Theorem}[section]
\usepackage{amsfonts,latexsym}
\newtheorem{lem}[thm]{Lemma}
\newtheorem{cor}[thm]{Corollary}

\newtheorem{prop}[thm]{Proposition}
\theoremstyle{definition}
\newtheorem{example}[thm]{Example}

\newtheorem{ques}[thm]{Question}
\newtheorem{rem}[thm]{Remark}
\newtheorem{rems}[thm]{Remarks}

\numberwithin{equation}{thm}
\usepackage{wrapfig}

\newcommand{\cx}{{\text{\rm cx}}}

\newcommand\Xreg{X^+_{\text{\rm reg}}}
\newcommand\Xregl{X^+_{\text{\rm reg},l}}

\newcommand{\opH}{\mbox{H}}

\newcommand{\Cx}{{\text{\rm Cx}}}

\newcommand{\Ext}{{\text{\rm Ext}}}

\newcommand{\Amod}{A\mbox{--mod}}
\newcommand{\Bmod}{B\mbox{--mod}}

\newcommand{\Hom}{\text{\rm Hom}}

\newcommand{\rad}{\operatorname{rad}}

\newcommand{\blist}{\begin{list}{\rom{(\roman{enumi})}}{\setlength
{\leftmarg in}{0em} \setlength{\itemindent}{7ex}
\setlength{\labelsep}{2ex}\setlength{\listparindent}{\parindent}
\usecounter{enumi}}}
\newcommand{\elist}{\end{list}}

\begin{document}

\title[Growth rates]{\large {\bf Cohomological growth rates and Kazhdan-Lusztig polynomials}}

\author{Brian J. Parshall}
\address{Department of Mathematics \\
University of Virginia\\
Charlottesville, VA 22903} \email{bjp8w@virginia.edu {\text{\rm
(Parshall)}}}
\author{Leonard L. Scott}
\address{Department of Mathematics \\
University of Virginia\\
Charlottesville, VA 22903} \email{lls2l@virginia.edu {\text{\rm
(Scott)}}}
\thanks{Research supported in part by the National Science
Foundation} 
\begin{abstract}In previous work, the authors established various bounds for the dimensions of degree
$n$ cohomology and $\Ext$-groups, for irreducible modules of semisimple algebraic groups $G$ (in positive characteristic $p$)  and (Lusztig) quantum groups $U_\zeta$ (at roots of unity $\zeta$).
These bounds depend only on the root system, and not on the characteristic $p$ or the size of the
root of unity $\zeta$. This paper investigates the rate of growth  of these bounds. Both in the quantum and algebraic group situation, these rates of growth represent new and fundamental invariants
attached to the root system $\Phi$. For quantum groups $U_\zeta$ with a fixed $\Phi$, we show the sequence
  $\{\max_{L\,{\text{irred}}}\dim \opH^n(U_\zeta,L)\}_n$ has polynomial
growth independent of $\zeta$. In fact, we provide upper and lower bounds for the polynomial growth rate. Applications
of these and related results for $\Ext^n_{U_\zeta}$ are given to Kazhdan-Lusztig polynomials.
Polynomial growth in the algebraic group case remains an open question, though it is proved that $\{\log\max_{L\,\text{\rm irred}}\dim\opH^n(G,L)\}$
has polynomial growth $\leq 3$ for any fixed prime $p$
(and $\leq 4$ if $p$ is allowed to vary with $n$).
 We indicate the relevance of these issues to (additional structure for) the constants proposed in the theory of higher cohomology groups for finite simple groups with irreducible coefficients by Guralnick, Kantor, Kassabov, and
Lubotzky \cite{GKKL}.

\end{abstract}

\maketitle
\section{Introduction}A sequence $\{s_n\}_{n=0}^\infty$ of complex numbers is said to have polynomial rate of growth if there
  exists a non-negative integer $d$ and a constant $C$ such that
\begin{equation}\label{rateofgrowth}|s_n|\leq C\cdot n^{d-1},\quad\forall n\gg 0.\end{equation}
When such a $d$ exists, the smallest is denoted $\gamma(\{s_n\})$ and is called
the (polynomial) rate of growth of the sequence.

In previous work \cite{PS5}, the authors established various bounds $s_n$ for the dimensions of degree $n$ cohomology
and Ext-groups, for irreducible modules of semisimple algebraic groups (in characteristic $p>0$) and quantum groups (at $l$th roots of unity). These bounds depend only on the
root system $\Phi$ (and not on the characteristic $p$ or the size of the root of unity). This paper addresses the issue of the rates of growth of these bounds.

There is a natural, broader context---``complexity theory"---for  rate of growth issues for cohomology or $\Ext$-groups. In a complexity
theory, one is interested in the rate of growth of cohomological invariants associated with individual modules rather than
with rates of growth of bounds on such invariants. The original example appears in the representation theory of finite groups in the work of Alperin and Evens \cite{AE}:
Let $k$ be an (algebraically closed) field. Let $G$ be a finite group and let $V$ be a finite dimensional $kG$-module.
Choosing a minimal projective resolution $\cdots\to P_n\to\cdots\to P_1\to P_0\twoheadrightarrow V$, the complexity $\cx(V)$
of $V$ was defined in \cite{AE} to be the rate of growth of the sequence $\{\dim P_n\}$. Equivalently, $\cx(V)$ equals
the rate of growth of the Ext-sequence $\{\dim\Ext^n_{kG}(V,V)\}$. As another interpretation, $\cx(V)$ equals the Krull dimension of the cohomological
support variety of $V$, the latter defined by a suitable homogeneous ideal in ${\opH}^{2\bullet}(G,k)$ (or in ${\opH}^\bullet(G,k)$,
if $k$ has characteristic 2).

The Alperin and Evens work \cite{AE} marked the start of a rich theory in representation theory.
In addition, the above definition of complexity led to  successful theories  in the category of
of finite dimensional modules
for a restricted enveloping algebra in positive characteristic,  for a finite group scheme $G$, and  for the little quantum
group $u_\zeta$ attached to a Lusztig quantum enveloping algebra $U_\zeta$ at a root of unity for a complex simple Lie
algebra, etc. For example, see \cite{FP}, \cite{Ost}, and \cite{FPe}, respectively.

If, however, $G$ is a connected, semisimple group over an algebraically closed field
$k$ of positive characteristic, then nonzero finite dimensional projective (or injective) rational $G$-modules do not
exist. A definition of the complexity of a finite dimensional rational $G$-module $V$ by means of the dimensions of a minimal
projective or injective resolution is thus not possible.\footnote{One possible alternative definition of the complexity
of $V$, equivalent in the case of
finite groups \cite{Ben}, is the rate of growth of the sequence defined by setting  $\gamma_m(V)=\max_{L\,{\text{irred}}}\dim\Ext^m_G(V,L)$. This sequence is discussed in \S6.} Furthermore, $\Ext^n_G(V,V)=0$ for $n\gg 0$, so that the rate
of growth of the Ext-sequence is uninteresting. Finally, the vanishing ${\opH}^{>0}(G,k)=0$ rules out an approach
via cohomological support varieties.

The situation for the quantum enveloping algebra $U_\zeta$, with $\zeta$ a root of unity, is similar, but more tractable. Here, every finite dimensional $U_\zeta$-module $V$
 has a finite dimensional projective cover. Also, $V$ has a minimal projective resolution $P_\bullet\twoheadrightarrow V$, so that $\gamma(\{\dim P^n\})$ makes sense. It is easy to see, however, that $P_\bullet$ is also a minimal projective resolution of
$V$ as a module for the little quantum group $u_\zeta$, so that this approach, while not uninteresting, just recovers the complexity
of $V|_{u_\zeta}$. As with the semisimple group $G$, $\Ext^n_{U_\zeta}(V,V)=0$ for $n\gg 0$. In addition, ${\opH}^\bullet(U_\zeta,{\mathbb C})=0$. The other two approaches thus fail to bear fruit.

It is nevertheless possible to introduce new notions of complexity for $U_\zeta$ and to establish quite strong upper and lower bounds concerning the rates of growth of dimensions of cohomology spaces
 for $U_\zeta$.  As shown in \cite[Thm. 6.5]{PS5}, the category $U_\zeta$-mod has remarkable (but elementary) cohomological boundedness properties:
 Given a non-negative integer $n$,  there is a positive constant $C'(\Phi,n)$ such that the following holds. If $U_\zeta$ is
 the (Lusztig) quantum enveloping algebra at an $l$th root of unity associated with the complex semisimple Lie algebra with root system $\Phi$,\footnote{It is required that $l>h$ (the Coxeter number) be odd and not divisible by $3$, if $\Phi$ has
a component of type $G_2$. See also footnote 5 below.}
 then, for any dominant weight $\lambda$,
\begin{equation}\label{start}\sum_{\nu}\dim\Ext^n_{U_\zeta}(L_\zeta(\lambda),L_\zeta(\nu))\leq C'(\Phi,n),\end{equation}
where the sum is over all dominant weights $\nu$.

The first main result of the present paper,
Theorem \ref{maintheorem}, establishes an upper bound of $|\Phi|$ on the rate of growth of the sequence $\{s_n\}$ defined as the minimal $C'(\Phi,n)$ that works in (\ref{start}). This result may be interpreted as giving a strong and universal (depending on $\Phi$, but not on $l$) upper bound of $|\Phi|$ on the complexity $\Cx(L_\zeta(\lambda))$ given
by the rate of growth of the left-hand side of (\ref{start}). Obviously, $\Cx(L_\zeta(\lambda))\geq\cx(L_\zeta(\lambda))$,
where the latter is the rate of growth of a sequence like that on the left-hand side of (\ref{start}), but with the sum over $\nu$
replaced by a maximum over $\nu$. We do not know if the upper bound of $|\Phi|$ on $\Cx(L_\zeta(\lambda))$ is sharp, but we
show, in Theorem \ref{lowerbound}, that $\cx(L_\zeta(\lambda))$ may be as large as $\text{rank}(\Phi)-1$ as long as $\text{rank}(\Phi)\geq 3$. (In Example \ref{example}, however, we
work out the precise bounds in case $\Phi$ has type $A_1$. Remarkably, even higher degree $\Ext$-groups between irreducible modules
are either 0- or 1-dimensional, and can be explicitly calculated; see Proposition \ref{last prop}.) For $u_\zeta$-complexity, as mentioned above, $|\Phi|$ is a sharp bound if $l>h$.

These results have natural consequences for Kazhdan-Lusztig polynomials, since the coefficients of certain of these
polynomials are precisely the dimensions of certain $\Ext$-groups for $U_\zeta$; see \cite[Thm. 3.9.1]{CPS1} . The polynomials in question are Kazhdan-Lusztig polynomials  for
the affine Weyl group $W_a$ of $\Phi$ which are parabolic with respect to the ordinary Weyl group $W$.\footnote{In our context,
this just means here
that they are Kazhdan-Lusztig polynomials
 $P_{w_0y,w_0x}$ for $W_a$, with $x,y\in W_a$ and  $\ell(w_0y)=\ell(w_0)+\ell(y)$ and $\ell(w_0x)=\ell(w_0)+\ell(x)$, where $w_0$ is the long word in $W$. These are the polynomials appearing in Lusztig's celebrated modular conjecture.} Indeed, the representations of quantum groups at
 a root of unity may be viewed as a conceptual fleshing out of the combinatorics of these difficult to compute polynomials. In particular, Theorem 4.1  draws on  the
results described above to bound above and below the rate of growth of certain sequences of coefficients of Kazhdan-Lusztig
polynomials. Section 5 explores an alternative approach to understanding these coefficients,
making use of the Koszulity results \cite{PS4} of certain $\Ext$-algebras associated with $U_\zeta$. Theorem 5.1
codifies some of this discussion in the language of a triangular decomposition of this Ext-algebra. We hope these
results will be of interest to combinatorial theorists as well as researchers in representation theory.

Finally, \S6 considers the case of a semisimple, simply connected algebraic group $G$ over an algebraically closed field $k$ and with a fixed root system
 $\Phi$. In many of our considerations, we will allow the field $k$ to vary and, in particular, to have different
 characteristic $p>0$. A previous result \cite[Thm. 7.1]{PS5}  established that there exists a constant, depending only on $\Phi$, and not on $p$, bounding all $\dim_{\text{\rm L irred}}\dim\opH^n(G,L)$.\footnote{As already observed in \cite[Rem. 7.5(a)]{PS5},
the sums analogous to (\ref{start}) are not necessarily bounded in the category of rational $G$-modules, though the dimension for $\lambda$ and $m$ fixed and $\nu$ arbitrary of an individual group $\Ext^m_G(L(\lambda),L(\nu))$  is bounded \cite[Thm. 7.1]{PS5}.  Also, a more restrictive result regarding sums does hold for $m=1$; see \cite[Thm. 5.4]{PS5}.} Theorem \ref{log} proves that
the sequence $\{\log\max_{L\,\text{\rm irred}}\dim {\opH}^n(G,L)\}$ (with $p$ allowed to vary) has polynomial growth $\leq 4$. (In fact, a somewhat
stronger result involving $\Ext$-groups is obtained.) Interestingly, if $p$ and $\Phi$ are fixed, this sequence has
rate of growth $\leq 3$.
However, Question
\ref{Gcase} asks whether $\{\max_{L\,\text{\rm irred}}\dim {\opH}^n(G,L)\}$ has polynomial
growth, depending only on $\Phi$ and not on $p$.  The cohomology groups $\opH^n(G,L)$ are relevant, by way of the theory
 of generic cohomology \cite{CPSK}, to the asymptotic study of higher cohomology
groups of finite simple groups with irreducible coefficients. Our considerations here suggest possible new structure to the constants
proposed in \cite[Ques. 12.1]{GKKL}; see the discussion above Question \ref{Gcase}.
Remark \ref{completeness} raises additional issues related to Question \ref{Gcase} and
points out a possible ``completeness" property, quite analogous
to completeness in the complexity theory of theoretical computer
science.
Section 6 closes with some further growth related questions regarding the multiplicative structure of $\Ext$-algebras associated with
families of irreducible $G$-modules.

The rates of growth considered in this paper represent fundamental
invariants attached to the root system $\Phi$ which, as far as the authors know, have never before been investigated. In fact, such issues make
no sense without the recent boundedness results of \cite{PS5}. Our results  generally address here only the issue of bounds for
 rates of growth,  but the authors hope
they will spark interest in this topic.
\section{Some notation and preliminaries} Let $\Phi$ be a finite and irreducible  root system, spanning a Euclidean space $\mathbb E$ with inner
product $(u,v)$, $u,v\in \mathbb E$. Fix a set $\Pi=\{\alpha_1,\cdots,\alpha_{\text{\rm{rank}},G}\}$ of simple roots, and let $\Phi^+$
be the corresponding set of positive roots.
Let $X^+\subseteq\mathbb E$ be the set of dominant weights; thus,
$\lambda\in\mathbb E$ belongs to $X^+$ if and only if $(\lambda,\alpha^\vee)\in{\mathbb N}$ (the non-negative integers),
for all $\alpha\in\Pi$. Here $\alpha^\vee=\frac{2}{(\alpha,\alpha)}\alpha$ is the coroot defined by $\alpha$.
The fundamental dominant weights $\varpi_1,\cdots,\varpi_r$ are the elements in $X^+$ defined by $(\varpi_i,\alpha_j^\vee)=
\delta_{ij}$, for $1\leq i,j\leq r$. Let $\rho=\varpi_1+\cdots+\varpi_r$ (Weyl weight). The weight lattice is defined
to be $X:={\mathbb Z}\varpi_1\oplus \cdots\oplus{\mathbb Z}\varpi_r\subset{\mathbb E}$. Let
$\alpha_0\in\Phi^+$ be the maximal short root, and let $h=(\rho,\alpha_0^\vee)+1$ be the Coxeter number. For $\lambda\in X^+$,
let $\lambda^\star:=-w_0(\lambda)\in X^+$, where $w_0$ is the element of maximal length in the Weyl group $W$ of $\Phi$.

Let $W\subset{\mathbb O}({\mathbb E})$ be the Weyl group of $\Phi$. For $\alpha\in\Phi$,  $s_\alpha\in W$
 is the reflection defined by $s_\alpha(u)=u-(u,\alpha^\vee)\alpha$, $u\in\mathbb E$. Setting $S:=\{s_\alpha\}_{\alpha\in\Pi}$,
$(W,S)$ is a Coxeter system. Given an integer $n$ and $\alpha\in\Phi$, let $s_{\alpha,n}:{\mathbb E}\to{\mathbb E}$
be the affine transformation given by $s_{\alpha,n}(u)=u-\left((u,\alpha^\vee)-n\right)\alpha$, $u\in\mathbb E$. Thus,
$s_\alpha=s_{\alpha,0}$. Let
$S_a:=S\cup\{s_{\alpha_0,-1}\}$, and let $W_a$ be the group of affine transformations on $\mathbb E$ generated by
the $s_{\alpha,n}$. Then $(W_a,S_a)$ is also a Coxeter system containing $W$ as a parabolic subgroup. For $(y,x)\in W_a\times W_a$,
let $P_{y,x}$ be the corresponding Kazhdan-Lusztig polynomial in a variable $t$; thus, $P_{1,1}=1$ and, if $x\not=y$, $P_{y,x}$ is a polynomial of degree $\leq \ell(x)-\ell(y)-1$. (In addition, $P_{y,x}$ is a polynomial in $q:=t^2$.)  Also, $P_{y,x}=0$ unless $y\leq x$ in the Bruhat-Chevalley ordering
on $W_a$.

Suppose that $l$ is a positive integer. Let $W_{a,l}$ be the subgroup of $W_a$ generated by the $s_{\alpha,n}$ with
 $n\equiv 0$ mod $l$. Thus, $W_{a,1}=W_a$. If $S_l=S\cup\{s_{\alpha_0,-l}\}$, $(W_{a,l},S_l)$ is a Coxeter system as
 well. Given $(y,x)\in W_{a,l}\times W_{a,l}$, let $P_{y,x,l}$ be the Kazhdan-Lusztig polynomial for this Coxeter system.  We will usually denote $P_{y,x,l}$ simply by $P_{y,x}$ since it will be
 clear from context which Coxeter group is being used. Also, the polynomials associated with
 $(y,x)$ with the ``same" expression
in terms of generating reflections are the same polynomials, and were even
originally defined for abstract Coxeter groups with a given set of generating reflections.

  The closure $\overline{C^-_l}$
 of the simplex $C^-_l:=\{x\in{\mathbb E}\,|\,(x+\rho,\alpha^\vee)<0, (x+\rho,\alpha_0^\vee)>-l\}$ is a fundamental domain for the ``dot" action of $W_{a,l}$ on $\mathbb E$, defined by
 $w\cdot x:=w(x+\rho)-\rho$.
 A weight $\lambda\in X$
is $l$-regular provided that, for all $\alpha\in\Phi$, $(\lambda+\rho,\alpha^\vee)\not\equiv 0$ mod$\,l$. The $l$-regular
weights all have the form $w\cdot\lambda^-$ for some $\lambda^-\in C^-_l\cap X$.  Such
weights exist if and only if $l\geq h$. Let $X_{\text{\rm reg},l}$ be the set of all $l$-regular weights, and
$\Xregl$ the subset of all $l$-regular weights which are dominant.

Each right coset $Wy$ of $W$ in $W_{a,l}$ has a unique element of maximal length. If $w_0\in W$ is the long word and if $y$
is the element of maximal length in $Wy$, then $w_0y$ is the unique element of minimal length in this right coset (usually,
one calls $w_0y$ a ``distinguished" right coset representative). Let $W_{a,l}^+$ be the set of longest right coset
representatives of $W$ in $W_{a,l}$. If $\lambda\in\Xregl$, then there exists unique $w\in W^+_{a,l}, \lambda^-\in C_l^-\cap X$
such that $\lambda=w\cdot\lambda^-$.

For any positive integer $l$, there is an isomorphism $\epsilon_l:W_a\overset\sim\to W_{a,l}$ in which $\epsilon_l(s_{\alpha,n})=s_{\alpha,nl}$. Clearly, $\epsilon_l$ maps $W^+_a$ onto $W^+_{a,l}$. If $x,y\in W_a\times W_a$,
$P_{y,x}=P_{\epsilon_l(y),\epsilon_l(x),l}$.

Let $l$ be an odd integer, not divisible by 3 in case $\Phi$ has type $G_2$. Let $\zeta=\sqrt[l]{1}$ be a primitive $l$th root of unity. Let ${\mathfrak g}_{\mathbb C}$ be the complex simple Lie algebra with root system $\Phi$. It universal enveloping
algebra is denoted $U=U({\mathfrak g}_{\mathbb C})$. In addition, $U_\zeta=U_\zeta({\mathfrak g}_{\mathbb C})$ denotes the (Lusztig) quantum enveloping
algebra associated with $\mathfrak g$ at the root of unity $\zeta$.\footnote{When speaking of the quantum enveloping algebra
$U_\zeta$, we will always assume
 $l$ is odd, and not divisible by $3$ in case $\Phi$ has a component of type $G_2$. In addition, a $U_\zeta$-module will always be taken to be of
 type 1 and integrable. We do not know if the conditions on $l$ can be eliminated in this paper. Such an elimination seems particularly likely
  in type $A$, since, for example, the theory of $q$-Schur algebras works nicely in that case.} The algebras $U_\zeta$ and $U$ admit the structure of
 Hopf algebras over $\mathbb C$, and they are related by the Frobenius morphism $F:U_\zeta\twoheadrightarrow U$ with Hopf kernel $u_\zeta=u_\zeta({\mathfrak g})$ (the ``little quantum group" of $U_\zeta$). Thus, $u_\zeta$ is a normal (Hopf)
subalgebra of $U_\zeta$, and $U_\zeta//u_\zeta\cong U$.

If $\lambda\in X^+$, $L_\zeta(\lambda)$ (resp.,
$L_{\mathbb C}(\lambda)$) is the irreducible $U_\zeta$ (resp., $U$) module of highest weight $\lambda$.
Given a (finite dimensional) $U$-module $V$, let $V^{(1)}=F^*V$ be its pull-back to $U_\zeta$ through
$F$. In particular, $L_{\mathbb C}(\lambda)^{(1)}\cong L_\zeta(l\lambda)$. Also, for $\lambda\in X^+$, let $\nabla_\zeta(\lambda)$ (resp.,
$\Delta_\zeta(\lambda)$) be the costandard (resp., standard) $U_\zeta$-module of highest weight $\lambda$.

If $M$ is a $U_\zeta$-module on which $u_\zeta$ acts trivially, then $M$ carries a natural $U$-module
structure. We denote this $U$-module by $M^{(-1)}$; it is an object in the category of locally finite
$U$-modules. Because this latter category is semisimple, an elementary Hochschild-Serre spectral sequence
argument shows that
\begin{equation}\label{spectralsequence}E^{a,b}_2 =\Hom_{U}({\mathbb C},{\opH}^n(u_\zeta,M)^{(-1)})\Rightarrow
{\opH}^n(U_\zeta,M),\end{equation}
for any $U_\zeta$-module $M$. A similar spectral sequence holds for  $\Ext^n_{U_\zeta}(M,N)\cong
{\opH}^n(U_\zeta, M^\star\otimes N)$ for $U_\zeta$-modules $M,N$ with $M$ finite dimensional.

Suppose $L,M$ are finite dimensional $U$-modules with $L$ irreducible. The following useful fact is used below. Its proof
is elementary; see \cite[Lemma 6.1]{PS5}. (Recall that $U=U({\mathfrak g}_{\mathbb C})$ is the universal enveloping
algebra of ${\mathfrak g}_{\mathbb C}$.)

\begin{lem}\label{usefulfact} The length (= number of composition factors) of the $U$-module
$L\otimes M$ is at most $\dim\,M$. \end{lem}

We will use below the fact that the dimensions of certain important $U_\zeta$-cohomology spaces are determined in
terms of Kazhdan-Lusztig polynomials. Suppose that $l>h$ is odd (and not divisible by 3 in case $\Phi$ has type $G_2$).
Fix some $\lambda^-\in C^-_l\cap X$.   Then for $x,y\in W^+_{a,l}$, letting $P_{y,x}$ be the Kazhdan-Lusztig polynomial
for the pair $(y,x)$ in the Coxeter group $W_{a,l}$, we have
\begin{equation}\label{KL1}
\begin{aligned}P_{y,x} &=t^{\ell(x)-\ell(y)}\sum_n\dim\Ext_{U_\zeta}^n(L_\zeta(x\cdot\lambda^-),\Delta_\zeta(y\cdot\lambda^-))t^{-n}\\
&=t^{\ell(x)-\ell(y)}\sum_n\dim\Ext^n_{U_\zeta}(\Delta_\zeta(y\cdot\lambda^-),L_\zeta(x\cdot\lambda^-))t^n.\end{aligned}
\end{equation}
While the above formulas show that certain important $\Ext$-spaces have dimensions determined by coefficients
of Kazhdan-Lusztig polynomials, there is also a converse effect: the representation theory of quantum groups
can viewed as an ``interpretation" of Kazhdan-Lusztig polynomials, placing their purely combinatorial definition
into a richer conceptual context. Indeed, we will show in this paper that properties of the representation
theory of the quantum groups $U_\zeta$ have consequences for Kazhdan-Lusztig polynomials.

One particularly fruitful aspect of this interaction begins with the following formula. For $x,y\in W^+_{a,l}$, and $n\geq 0$,
\begin{equation}\label{KL2}\begin{aligned}
\dim&\Ext^n_{U_\zeta}(L_\zeta(x\cdot\lambda^-),L_\zeta(y\cdot\lambda^-))\\ &=\sum_{z\in W^+_{a,l}}\sum_{a+b=n}\dim\Ext_{U_\zeta}^a(L_\zeta(x\cdot
\lambda^-),\nabla_\zeta(z\cdot\lambda^-))\cdot\dim\Ext^b_{U_\zeta}(\Delta_\zeta(z\cdot\lambda^-),L_\zeta(y\cdot\lambda^-)).
\end{aligned}\end{equation}
The Lusztig character formula holds for $U_\zeta$ (see, e.~g., \cite[\S7]{T}), so  \cite[Thm. 3.9.1]{CPS1} is applicable.\footnote{Unfortunately, no such formula is known for singular dominant weights. However, it is
still possible to obtain upper bounds using a similar expression (when $l\geq h$), using derived category methods (see the earlier
 version of \cite{PS5} posted on arXiv) or translation
operators.}
Making use of the isomorphism $\epsilon_l:W^+_a\overset\sim\to W^+_{a,l}$, it follows that, as long as attention is
restricted to $l$-regular weights, questions about bounds on the dimensions of the $\Ext^n_{U_\zeta}$-spaces above
are largely independent of $l$. In particular, it is convenient to check things for a particular $l$, and then these results hold for all $l$. For example, the case $l=p$, a prime, is useful, because the representations of $U_\zeta$ are then related to those
of the simple algebraic group $G$ of the same root type in characteristic $p$.

\section{Complexity} Define the complexities $\cx(U_\zeta)$ and $\text{Cx}(U_\zeta)$ by
\begin{equation}\label{complex}
\begin{cases} \cx(U_\zeta):=\gamma(\{\max_{\lambda,\nu\in X^+}\dim\Ext^n_{U_\zeta}(L_\zeta(\lambda),L_\zeta(\nu))\});\\
\Cx(U_\zeta)=\gamma(\{\max_{\lambda\in X^+}\sum_{\nu\in X^+}\dim\Ext^n_{U_\zeta}(L_\zeta(\lambda),L_\zeta(\nu))\}).
\end{cases}\end{equation}
As defined, $\cx(U_\zeta)$ and $\Cx(U_\zeta)$ might depend on the choice of $l$ (i.~e., $\zeta$). Certainly,
$\cx(U_\zeta)\leq \Cx(U_\zeta)$.\footnote{An appropriate definition of the complexity $\cx(M)$ of a finite dimensional $U_\zeta$-module $M$
might be $\cx(M):=\gamma(\{\max_{\lambda\in X^+}\dim\Ext^n_{U_\zeta}(M,L(\lambda))\})$. We do not develop this notion
in any detail in this paper, though we consider similar issues for rational $G$-modules in \S6.  In both the $G$-module
and quantum case, we are more
concerned in this paper with
using the context of module complexity
to obtain bounds on growth rates,
rather than analyzing complexities
of individual modules.}

In the following theorem, we will call an integer $l>1$ {\it exceptional} for the root system $\Phi$ provided one
of the following (redundant) conditions holds:

\begin{itemize}
\item[(i)] $l$ is even, or divisible by $3$ in case $\Phi$ has type $G_2$.

\item[(ii)] $l$ is a bad prime for $\Phi$. In other words, $l=2$ in types $B,C,D$; $2,3$ in types $E_6,E_7,F_4, G_2$; and
$l=2,3,5$ in type $E_8$.

\item[(iii)] $l$ divides $n+1$ in type $A_n$.

\item[(iv)] $l=9$ in type $E_6$ and $l=7$ or $9$ in type $E_8$.

\end{itemize}

Note that no odd $l>h$ is exceptional. Now, given $l>1$, define $\Phi_{0,l}:=\{\alpha\in\Phi\,|\,(\rho,\alpha^\vee)\equiv 0\,\mod l\}$. When $l$ is not exceptional,
we have that ${\opH}^\bullet(u_\zeta,{\mathbb C})={\opH}^{2\bullet}(u_\zeta,{\mathbb C})$ is a (reduced) finitely generated $k$-algebra
(in fact, an integral domain) of Krull dimension equal to $|\Phi\backslash\Phi_{0,l}|$. This is proved in \cite{GK} when $l>h$ and
in \cite[Thm. 1.2.3]{BNPP} when $l\leq h$. In fact, ${\opH}^\bullet(u_\zeta,{\mathbb C})$ is the coordinate algebra of the
nullcone of $\mathfrak g$, when $l>h$. In this case, $\Phi_{0,l}=\emptyset$ (which holds even if $l=h$). When $l\leq h$, ${\opH}^\bullet(u_\zeta,{\mathbb C})$ is the coordinate algebra of
an explicitly determined affine variety of the form $G\cdot {\mathfrak u}_J$, where ${\mathfrak u}_J$ is the nilpotent
radical of the parabolic subalgebra of $\mathfrak g$ (containing the negative Borel subalgebra) for a suitable subset
$J\subseteq\Pi$. In addition, if $M$ is a finite dimensional $u_\zeta$-module, $\Ext^\bullet_{u_\zeta}(M,M)$ is a finitely
generated ${\opH}^\bullet(u_\zeta,{\mathbb C})$-module. (See also \cite{MPSW}.)

In the following result, the reader should keep in mind that $\Phi_{0,l}=\emptyset$ if $l\geq h$.
 \begin{thm}\label{maintheorem}There exists a constant $C(\Phi,l)$, depending only on $\Phi$, such that, for all $l>1$ which are
 not exceptional, and all $\lambda\in X^+$, we have
 \begin{equation}\label{display}\sum_{\nu\in X^+}\dim\Ext^n_{U_\zeta}(L_\zeta(\lambda),L_\zeta(\nu))\leq C(\Phi,l)n^{|\Phi\backslash\Phi_{0,l}|-1}.\end{equation}
Furthermore, if $l>h$, the constant $C(\Phi,l)$ can be chosen independently of $l$.\footnote{Thus, there is a constant
independent of $l$ which serves for $C(\Phi,l)$ for all non-exceptional $l$ in (\ref{display}). The exponent $|\Phi\backslash\Phi_{0,l}|$ can change, however, when $l<h$, though it is always at most $|\Phi|$.}
\end{thm}

\begin{proof} First, assume that $l $ is fixed and that $l$ is not exceptional. For any $\tau\in X^+$, write $\tau=\tau_0+l\tau_1$, where
$\tau_0,\tau_1\in X^+$ and $\tau_0$ is $l$-restricted (i.~e., $(\tau_0,\alpha^\vee)<p$, $\forall \alpha\in\Pi$). Thus, $L_\zeta(\lambda)\cong L_\zeta(\lambda_0)
\otimes L_{\mathbb C}(\lambda_1)^{(1)}$.
Then
$$\begin{aligned} \sum_\nu\dim\Ext^n_{U_\zeta}(L_\zeta(\lambda),L_\zeta(\nu)) &=\sum_\nu\dim\Hom_{U}(L_{\mathbb C}(\lambda_1), L_{\mathbb C}(\nu_1)\otimes\Ext^n_{u_\zeta}(L_\zeta(\lambda_0),L_\zeta(\nu_0))^{(-1)})\\
&=\sum_{\nu}\dim\Hom_{U}(L_{\mathbb C}(\nu_1^\star), L_{\mathbb C}(\lambda_1^\star)\otimes\Ext^n_{u_\zeta}(L_\zeta(\lambda_0),L_\zeta(\nu_0))^{(-1)})\\
&\leq \sum_{\nu_0}\dim\Ext^n_{u_\zeta}(L_\zeta(\lambda_0),L_\zeta(\nu_0))\quad{\text{\rm (by Lemma \ref{usefulfact})}}\\
&\leq C(\Phi,l)n^{|\Phi\backslash\Phi_{0,l}|-1}.\end{aligned}$$
The second equality holds since, for any $\tau\in X^+$,  $L_{\mathbb C}(\tau^\star)$ is isomorphic
to the dual $L_{\mathbb C}(\tau)^*$. In the last inequality, we use the fact that $\Ext^\bullet_{u_\zeta}(L_\zeta(\lambda),\bigoplus_{\nu_0}L_\zeta(\nu_0))$ is
a graded submodule of the finitely generated graded module $\Ext_{u_\zeta}^\bullet(\bigoplus_{\nu_0}L_\zeta(\nu_0),\bigoplus_{\nu_0}L_\zeta(\nu_0))$ for ${\opH}^\bullet(u_\zeta,{\mathbb C})$.
It therefore has rate of growth $\gamma(\{ \dim\Ext^n_{u_\zeta}(\bigoplus_{\nu_0}L_\zeta(\nu_0),\bigoplus_{\nu_0}L_\zeta(\nu_0))\})$ at most the Krull dimension of ${\opH}^\bullet(u_\zeta,{\mathbb C})$, namely, $|\Phi\backslash\Phi_{0,l}|$; cf. \cite[\S5.4]{Ben}.

This completes the proof of the first assertion of the theorem. For the second assertion, assume that $l>h$. We must show
that the constants $C(\Phi,l)$ can be chosen independently of $l$. For  $l$-regular dominant weights $\lambda,\nu$, $\dim\Ext^n_{U_\zeta}(L_\zeta(\lambda),L_\zeta(\nu))$ is given by a sum of products of Kazhdan-Lusztig
polynomials, and so it is independent of $l$. Therefore, there is a constant $C(\Phi)$ such that, for all $l$,
$\sum_\nu\dim\Ext^n_{U_\zeta}(L_\zeta(\lambda),L_\zeta(\nu))\leq C(\Phi)n^{|\Phi|-1}$ for all sufficiently large $n$.

Now suppose that $\nu$ is $l$-singular (for $l$ arbitrary). Then $\nu$ belongs to the upper closure of a dominant
alcove containing an $l$-regular dominant weight $\nu'$. Let $T=T_{\nu'}^\nu$ be the translation operator from
the category of $U_\zeta$-modules in the block of $L_\zeta(\nu')$ to those in the block of $L_\zeta(\nu)$. Hence, $TL_\zeta(\nu')=L_\zeta(\nu)$.
If $T'=T^{\nu'}_\nu$ is its adjoint translation operator, we have
$$\Ext^n_{U_\zeta}(L_\zeta(\lambda),L_\zeta(\nu))\cong\Ext^n_{U_\zeta}(L_\zeta(\lambda),TL_\zeta(\nu'))\cong\Ext^n_{U_\zeta}(T'L_\zeta(\lambda),L_\zeta(\nu')).$$
On the other hand, $T'L_\zeta(\lambda)$ has an irreducible head, so it is a homomorphic image of the $U_\zeta$-projective
cover $Q_\zeta(\xi)$ of $L_\zeta(\xi)$, where $\xi$ is the $l$-regular weight in the unique alcove containing $\lambda$ in its upper
closure which is linked to $\nu'$. By \cite[Lemma 6.3]{PS5}, there is a uniform bound $e(\Phi)$ on the length of $Q_\zeta(\xi)$.
Hence,
$$\sum_\nu\dim\Ext^n_{U_\zeta}(L_\zeta(\lambda),L_\zeta(\nu))\leq e(\Phi)C(\Phi)n^{|\Phi|-1},$$
as required.
The above argument completes the proof in case $l>h$. \end{proof}

\begin{cor}\label{complexityuppper} For $l$ not exceptional, $\Cx(U_\zeta)\leq |\Phi\backslash\Phi_{0,l}|$. In particular, if $l>h$,  there exists a constant $C_0(\Phi)$ such that for all regular dominant weights $\nu$
$$\dim {\opH}^n(U_\zeta,L_\zeta(\nu))\leq C_0(\Phi)n^{|\Phi|-1}.$$\end{cor}.

Now let $G$ be a simple, simply connected algebraic group over an algebraically closed field $k$ (of arbitrary characteristic) which has
root system $\Phi$ with respect to a maximal torus $T$. Assume $B\supset G$ is the Borel subgroup corresponding to
  the negative roots, and let ${\mathfrak u}$ be the Lie algebra of the unipotent radical of $B$. Let $S^\bullet({\mathfrak u}^*)$
  be the symmetric algebra of ${\mathfrak u}^*$, and for any integer $n\geq 0$ and weight $\sigma\in X= X(T)$, let $S^n({\mathfrak u}^*)_\sigma$
  be the $\sigma$-weight space in $S^n({\mathfrak u}^*)$. Define
  \begin{equation}\label{symmetric complexity}
  {\mathfrak s}(\Phi):=\gamma(\{\max_\sigma\dim S^n({\mathfrak u}^*)_\sigma\}).\end{equation}
Of course, ${\mathfrak s}(\Phi)$ is independent of the field $k$.

 Now we can prove the following result, giving a lower bound on complexity.
 \begin{thm}\label{lowerbound} Assume $l>h$.

 (a) We have
 $$\cx(U_\zeta)\geq {\mathfrak s}(\Phi).$$

 (b) Moreover,
 \begin{itemize}
 \item[(i)] ${\mathfrak s}(\Phi)= 1$ for types $A_1,A_2$;
 \item[(ii)] ${\mathfrak s}(\Phi)\geq 2$ for type $B_2$;
 \item[(iii)] ${\mathfrak s}(\Phi)\geq 4$ for type $G_2$;
 \item[(iv)] ${\mathfrak s}(\Phi)\geq r-1$ if the rank $r$ of $\Phi$ is at least 3.
 \end{itemize}

 \end{thm}
  \begin{proof}Assume the group $G$ above is taken over an algebraically closed field $k$ of sufficiently large
   characteristic $p$ so that the Lusztig character formula holds for all irreducible rational $G$-modules $L(\lambda)$
   having highest weight in the Jantzen region consisting of all $\lambda\in X^+$ satisfying the condition that $(\lambda+\rho,\alpha_0^\vee)\leq p(p-h+2)$.
See \cite{Fiebig} for specific bounds on $p$ (which depend on the root system $\Phi$). Assume also that $G$ is defined and split over ${\mathbb F}_p$, and that $T$
is a maximal split torus, contained in the Borel subalgebra $B$ corresponding to the set $-\Phi^+$ of negative roots.
  For any $w\in W$, let $\xi\in X^+$ and let
 $\mu:=w\cdot 0+p\xi\in X^+$. Let $G_1$ be the first infinitesimal subgroup of $G$ (i.~e., the kernel of the Frobenius
 morphism $G\to G$). Then, by \cite{KLT},
$${\opH}^0(G/B,S^{\frac{n-\ell(w)}{2}}({\mathfrak u}^*)\otimes\xi)\cong {\opH}^n(G_1,\nabla(\mu))^{(-1)}.
$$
In this expression, the right-hand side is interpreted to be 0 if $n-\ell(w)\not\in 2{\mathbb Z}$.
To make an explicit choice of $\mu$, take $w=1$ and assume that $n=2m$ is an even integer. By taking $\xi\in X^+$ sufficiently
large, we can assume that all the weights in  $S^m({\mathfrak u}^*)\otimes\xi$ are dominant. Let $\Xi\subset X^+$
be the set of all these dominant weights. Then, using the Kempf vanishing theorem, ${\opH}^0(G/B,S^m({\mathfrak u}^*)\otimes \xi)$ has a filtration with sections
of the form $\nabla(\tau)={\opH}^0(G/B,\tau)$, $\tau\in\Xi$. Furthermore, if $d(\tau)$ is the multiplicity of $\tau$ as a weight in $S^m({\mathfrak u}^*)\otimes\xi$, then $\nabla(\tau)$ appears as a section exactly $d(\tau)$ times in the corresponding good filtration of
${\opH}^n(G_1,\nabla(\mu))$. On the other hand, by the Hochschild-Serre spectral sequence
$$\dim\Ext^n_G(\Delta(\tau)^{(1)},\nabla(\mu))=\dim\Hom_G(\Delta(\lambda), {\opH}^n(G_1,\nabla(\mu))^{(-1)})=d(\tau).$$
Here we use the fact that $$\Ext^i_G(\Delta(\tau), {\opH}^{n-i}(G_1,\nabla(\mu))=0$$ for $i>0$ because $\dim\Ext^i_G(\Delta(\tau),
\nabla(\sigma))=\delta_{i,0}\delta_{\tau,\sigma}$.

Now let $\zeta=\sqrt[p]{1}$ and consider the quantum group $U_\zeta$. By \cite[Cor. 5.2]{CPS7},
$$
\dim\Ext^n_{G}(\Delta(\tau)^{(1)},\nabla(\mu))=\dim\Ext^n_{U_\zeta}(L_\zeta(p\tau),\nabla_\zeta(\mu)).$$
In addition, the inclusion $L_\zeta(\mu)\hookrightarrow \nabla_\zeta(\mu)$ induces a surjective map
\begin{equation}\label{surjective condition} \Ext^n_{U_\zeta}(L_\zeta(p\tau),L_\zeta(\mu))\twoheadrightarrow\Ext^n_{U_\zeta}(L_\zeta(p\tau),\nabla_\zeta(\mu)),\end{equation}
using \cite[Thm. 4.3]{CPS1} and the fact that the full subcategory of $U_\zeta$-mod with composition factors $p$-regular weights
has a Kazhdan-Lusztig theory. Putting things together, we get that
$$\dim \Ext^n_{U_\zeta}(L_\zeta(p\tau),L_\zeta(\mu))\geq
 d(\tau).$$
We can take $\sigma$ so that $S^m({\mathfrak u}^*)_\sigma$ is the weight space  
 in $S^m({\mathfrak u}^*)$ of maximal dimension, and set $\tau:=\sigma+\xi$. Then
$$
\dim \Ext^n_{U_\zeta}(L_\zeta(p\tau),L_\zeta(\mu))\geq {\mathfrak s}(\Phi).$$
Since the weights involved on the left-hand side are all $p$-regular, these dimensions are given in terms of the coefficients of Kazhdan-Lusztig
polynomials (see (\ref{KL1}) and (\ref{KL2})), and therefore are independent of $l$, using an appropriate $W_a^+$-parametrization. Thus, for any such $l$, $\cx(U_\zeta)\geq {\mathfrak s}(\Phi)$ as required for (a).

Now we prove (b).  Part (i) is obvious since each $\sigma$-weight space in $S^n({\mathfrak u}^*)$ is at most 1-dimensional.
We prove (ii). Write $\Pi=\{\alpha,\beta\}$ with $\beta$ short. For non-negative integers $A,B$, $\sigma=A\alpha+B\beta$ is
a weight of $S^n({\mathfrak u}^*)$ if and only if there exist non-negative integers $a,b,c$ such that
$$ \begin{aligned} A= a+c+d;\\
                   B=b+c+2d;\\
                   n=a+b+c+d.
                   \end{aligned}
                   $$
In particular, if $n=A=B$, then $b=0,a=d$ and $c=n-2d$. Thus, $d$ can assume any integer value between $0$ and $[n/2]$,
 and (ii) follows. Part (iii) is similar, but more complicated, using
 similar equations for $A$, $B$ in  $\sigma=A\alpha+B\beta$ ($\beta$ short), taking $A=B=n$ as in type $B_2$. We details
 to the reader.

 Finally, the proof of (iv) uses Lemma \ref{combinatorialresult} below. It says, for $\Phi$ having
 rank $\geq 3$, that given any positive integer $M$, there exists a integer $m>M$
such that, for some weight $\sigma$, $d(\sigma)>d(\Phi)m^{r-2}$. Then taking $n=2m$ and choosing $\mu$ as above, we
get that
$$\dim\Ext^n_{U_\zeta}(L_\zeta(p\sigma),L_\zeta(\mu))\geq d(\Phi)m^{r-2}.$$
Since the weights involved are all $p$-regular, these dimensions are given in terms of the coefficients of Kazhdan-Lusztig
polynomials, and therefore are ``the same" using an appropriate $W_a^+$ parametrization. Thus, for any such $l$, $\cx(U_\zeta)\geq r-1$, as required.
  \end{proof}

\begin{lem}\label{combinatorialresult} Assume the rank of $\Phi$ is at least 3. There exists a constant $d(\Phi)>0$ with the following property. For any positive
number $M$, there exists an integer $m>M$ and a weight $\tau$ such that, letting $r$ be the rank of $\Phi$,
$$\dim S^m(u^*)_\tau\geq d(\Phi)m^{r-2}.$$
\end{lem}

\begin{proof}Fix a linear ordering on the set $\Pi$ of simple roots. Let ${\mathcal T}$ be the set of all triples $(\alpha,\beta,\gamma)$ of distinct simple roots such that $\gamma>\alpha$ and
$(\alpha,\beta)\not=0\not=(\beta,\gamma)$. Thus, $\alpha+\beta$, $\beta+\gamma$ are roots. For each triple
$\tau=(\alpha,\beta,\gamma)\in{\mathcal T}$, put $\tau^\diamond=\alpha+2\beta+\gamma$. There are two ways to
write $\tau^\diamond$ as a sum of three positive roots:
\begin{itemize}
\item[(i)] $\tau^\diamond= (\alpha+\beta) +\beta+\tau$;
\item[(ii)] $\tau^\diamond=\alpha+\beta+ (\beta+\gamma)$.
\end{itemize}
Another way to say this is that we have two ``partitions" ${\mathfrak p}_L(\tau^\diamond)$, ${\mathfrak p}_R(\tau^\diamond)$
consisting of
$${\mathfrak p}_L(\tau^\diamond)=\{\alpha+\beta,\beta,\gamma\},\,{\mathfrak p}_R(\tau^\diamond)=\{\alpha,\beta,\beta+\gamma\}.$$
Clearly,
$$|{\mathcal T}|=\begin{cases}  r-2, \quad {\text{\rm in types $A,B,C,F,G$}};\\
r-1,\quad{\text{\rm otherwise}}.\end{cases}.$$
Put $$\tau^\Phi=\sum_{\tau\in{\mathcal T}}\tau^\diamond\in{\mathbb N}\Pi.$$

Fix a positive integer $n$. For each $\tau\in\mathcal T$ choose a non-negative integer $n_\tau\leq n$,  and form the sequence $\underline n=(n_\tau)_\tau$.
Define a partition
\begin{equation}\label{list}{\mathfrak p}_{\underline n}(\tau^\Phi):=\sum_{\tau} n_\tau{\mathfrak p}_L(\tau^\diamond)+(n-n_\tau){\mathfrak p}_R(\tau^\diamond).\end{equation}
By this we mean an unordered list of terms each of which is a root coming from the lists ${\mathfrak p}_L(\tau^\diamond)$ and ${\mathfrak p}_R(\tau^\diamond)$. (A simple root $\alpha$ can appear several times in this list.)

We claim these unordered lists of roots
are all distinct. We give the proof in the basic case when $\Phi$ has type $A,B,C$ and $F$ (essentially all the same). We
leave the other cases to the reader.  Suppose we are given an unordered list ${\mathfrak p}={\mathfrak p}_{\mathcal S}$ in
(\ref{list}), we wish to show that $\mathcal S$ is determined by it. Without loss, we can assume the ordering on
$\Pi$ is from left to right in the Dynkin diagram of $\Phi$. The triples in $\mathcal T$ can be enumerated $\tau_1=(\alpha_1,\alpha_2,
\alpha_3), \tau_2=(\alpha_2,\alpha_3,\alpha_4), \cdots, \tau_{r-2}=(\alpha_{r-2},\alpha_{r-1},\alpha_r)$. Given a list
$\mathfrak p$ in (\ref{list}), we wish to show that the sequence $\underline n$ is determined uniquely by it. First,
$n_{\tau_1}$ is just the number of occurrences of $\alpha_1+\alpha_2$ in $\mathfrak p$.  Suppose inductively that
we have determined $n_{\tau_1},\cdots, n_{\tau_m}$ for $m<r-2$. We wish to determine if $n_{\tau_{m+1}}$.
However, the number of times $\alpha_m+\alpha_{m+1}$ occurs in $\mathfrak p$ is $(n-n_{\tau_m})+n_{\tau_{m+1}}$.

Each partition $\mathfrak p$ contains $m=3n|{\mathcal T}|$ roots which all sum to $n\tau^{\Phi}$. There are
$(n+1)^{|{\mathcal T}|}$ such partitions. But $|{\mathcal T}|\geq r-2$ (with equality, except in types $D_n$, $n\geq 4$, $E_6$, $E_7$, $E_8$). Thus,
$$\begin{aligned}\dim S^m({\mathfrak u}^*)_{n\tau^{\Phi}} &\geq (n+1)^{r-2}\\
&> \left(\frac{1}{3|{\mathcal T}|}\right)^{r-2} m^{r-2},\end{aligned}$$
as required.
\end{proof}

\begin{rem} The lower estimate provided by Theorem \ref{lowerbound} is almost certainly too low in most cases, but at least
shows that $\cx(U_\zeta)\to \infty$ as $r\to \infty$. (Here $\Phi$ is irreducible.)
A better guess bound for a lower bound would be $|\Phi^+\backslash\Pi|=|\Phi^+|-r$, when $r>1$, as occurs about for
$A_2,B_2$ and $G_2$.   We work out the exact situation below for type $A_1$. The reader should observe that complete
knowledge of the dimensions of weight spaces of $S^\bullet({\mathfrak u})$ would not suffice to calculate $\cx(U_\zeta)$,
since the rate of growth of these weight spaces gives only a lower bound.
 \end{rem}

\begin{example}\label{example} Consider the case ${\mathfrak g}_{\mathbb C}={\mathfrak{sl}_2}({\mathbb C})$ and $l>2$. (The
  calculation also works for $l=2$.) Thus, $W_{a,l}$ is the infinite dihedral group on two generators $s_{\alpha}, s_{\alpha,-l}$. In this
case, $W^+_{a,l}$ consists of elements $(s_\alpha s_{\alpha,-l})^m, s_{\alpha}(s_{\alpha,-l}s_\alpha)^m$, $m\geq 0$. (For any positive integer $n$, there are two elements of $W^+_{a,l}$ of length $n$, but only one is dominant.) Hence,
for any positive integer $n$, $W^+_{\alpha,l}$ contains a unique element of length $n$. Also, if $x,y\in W_{a,l}$,
$x<y\iff \ell(x)<\ell(y)$. The Kazhdan-Lusztig polynomials are easy to describe. For $x,y\in W_{a,l}$, $P_{x,y}=1\iff x\leq y$; otherwise, $P_{x,y}=0$.

If $\lambda\in X^+$ is an $l$-singular weight, then it is easy to see that $L_\zeta(\lambda)$ is projective/injective, so there
is no nonzero cohomology in positive degree for such modules. In considering $l$-regular dominant weights, we can restrict
to the special case of dominant weights of the form $\lambda=x\cdot(-2\rho)$, $x\in W^+_{a,l}$. Denote $L_\zeta(\lambda)$, $\Delta_\zeta(\lambda)$,
 or $\nabla_\zeta(\lambda)$ simply by $L_\zeta(x)$, $\Delta_\zeta(x)$, $\nabla_\zeta(x)$, respectively.
By (\ref{KL1}), given $x\geq z$ in $W^+_{a,l}$,
\begin{equation}\label{specific}\begin{aligned}
\dim\Ext^n_{U_\zeta}(\Delta_\zeta(z),L_\zeta(x)) &=\dim\Ext^n_{U_\zeta}(L_\zeta(x),\nabla_\zeta(z))\\
&=
\begin{cases} 1\,\, \,\,n=\ell(x)-\ell(z);\\
0\,\,\,\,{\text{\rm otherwise}}.\end{cases}\end{aligned}\end{equation}

Furthermore, suppose that $a,b\in\mathbb N$ and $z\in W^+_{a,l}$, and that $$\dim\Ext^a_{U_\zeta}(L_\zeta(x),\nabla_\zeta(z))\dim\Ext^b_{U_\zeta}(\Delta_\zeta(z),L(y))\not=0.$$
Then (\ref{specific}) implies, after eliminating $\ell(z)$ and setting $n=a+b$, that
\begin{equation}\label{conditions}
\begin{cases}a=\frac{n+\ell(x)-\ell(y)}{2}\,\,{\text{\rm and}}\,\,b=\frac{n-\ell(x)+\ell(y)}{2};\\
\ell(z)=\frac{\ell(x)+\ell(y)-n}{2};\\
\ell(x)-\ell(y)\equiv n\,\mod 2.\end{cases}\end{equation}
Because $a,b$ and also $z$ are uniquely determined, if $\Ext^n_{U_\zeta}(L_\zeta(x),L_\zeta(y))\not=0$,
then it is 1-dimensional.

\begin{center}

\begin{picture}(150,150)

\put(20,0){\vector(0,1){150}}

\put(0,30){\vector(1,0){150}}

\put(20,80){\line(1,1){60}}

\put(20,80){\line(1,-1){50}}

\put(70,30){\line(1,1){60}}

\put(7,78){$n$}

\put(67,15){$n$}

\put(25,140){$Y$}

\put(140,15){$X$}

\put(60,80){$\Omega(n)$}

\end{picture}

\end{center}

 Conditions (\ref{conditions}) are also sufficient for non-vanishing as long as the integers involved are all non-negative. More precisely, let $x,y\in W^+_{a,l}$ and $n\in\mathbb N$ be so that $\ell(x)-\ell(y)\equiv 0$ mod$\,2$. Put $X=\ell(x), Y=\ell(y)$. Then $\Ext^n_{U_\zeta}(L_\zeta(x),L_\zeta(y))\not=0$ if and only if
$|X-Y|\leq n\leq X+Y $ and $X+Y\equiv n \,\,\mod 2$. Thus, non-vanishing means that $(X,Y)\in{\mathbb N}\times{\mathbb N}$ must lie
in the strip in ${\mathbb R}^2$ bounded by the lines $X- Y=n$ and $Y-X=n$, and additionally lie on or
to the right  of the line $X+Y=n$. Let $\Omega(n)$ denote the set of such integral pairs $(X,Y)$ with $X-Y\equiv n\,\mod 2$. (See
the figure above.) We have proved
\begin{prop}\label{last prop}Let $n$ be a non-negative integer. For $x,y\in W^+_{a,l}$, put $X=\ell(x),Y=\ell(y)$. Then
$$\dim\,\Ext^n_{U_\zeta}(L_\zeta(x),L_\zeta(y))=\begin{cases} 1,\quad\,\, (X,Y)\in\Omega(n);\\ 0,\quad\,\,{\text{\rm otherwise}}.\end{cases}
$$
\end{prop}
Therefore, $\cx(U_\zeta)=1$, and the bound in Theorem \ref{lowerbound} is sharp.
For a fixed $n$ and $x\in W_{a,l}^+$, the maximum value of $\sum_y\dim\Ext^n_{U_\zeta}(L_\zeta(x),L_\zeta(y))$ is thus $n+1$, and it is achieved for any $x$ satisfying $X=\ell(x)\geq n$. Therefore,
$\Cx(U_\zeta)=2=|\Phi|$, and the bound in Theorem \ref{maintheorem} in sharp.
 \end{example}
\begin{rem} Once the ``answer" region $\Omega(n)$ is known in the above example, it is not difficult to give an
alternative proof of Proposition \ref{last prop} by induction on $n$, using the exact sequences
$0\to L_\zeta(x)\to \nabla_\zeta(x)\to L_\zeta(x')\to 0$, where $\ell(x')+1=\ell(x)$, together with equation
(\ref{specific}).\end{rem}

\section{Bounding Kazhdan-Lusztig polynomial coefficients} The results on complexity established above have immediate consequences for Kazhdan-Lusztig polynomial coefficients. In this short section, we write these implications down.

 For $x,y\in W^+_a$ and a non-negative integer $n$, we let $c^{[n]}_{y,x}$ be
the coefficient of $t^n$ in the Kazhdan-Lusztig polynomial $P_{y,x}$. Recall from \S2, that we regard $P_{y,x}$ as a polynomial in $t$,
 rather than $q=t^2$. By (\ref{KL1}), given any $x,y\in W_a^+$,
$$c^{[\ell(x)-\ell(y)-n]}_{y,x}=\dim\Ext^n_{U_\zeta}(L_\zeta(\epsilon_l(x)\cdot\lambda^-),\nabla_\zeta(\epsilon_l(y)\cdot\lambda^-)),$$
for any $\lambda^-\in C^-_l\cap X$ and $l>h$. (Recall the isomorphism $\epsilon_l:W_a\overset\sim\to W_{a,l}$ defined in
\S2, which maps $s_{\alpha,n}$ to $s_{\alpha,ln}$.)
In addition,
$$\dim\Ext^n_{U_\zeta}(L_\zeta(\epsilon_l(x)\cdot\lambda^-),L_\zeta(\epsilon_l(y)\cdot\lambda^-))\geq \dim\Ext^n_{U_\zeta}(L_\zeta(\epsilon_l(x)\cdot\lambda^-),\nabla_\zeta(\epsilon_l(y)\cdot\lambda^-)).$$
Also, using (\ref{KL2}),
by (\ref{start}),
\begin{equation}\label{KLbound}C^{[n]}:=\max_{x,y\in W^+_a} c^{[\ell(x)-\ell(y)-n]}_{x,y}<\infty.\end{equation}

We can therefore reinterpret the results of the previous section to obtain the following result
involving bounds for Kazhdan-Lusztig polynomials of elements in $W^+_a$. In the case of the lower
bound, we need to use the dimension on the right-hand side of (\ref{surjective condition}) which is a Kazhdan-Lusztig
coefficient.

\begin{thm}Assume that $\Phi$ has rank at least 3. Then
$$r-1\leq \gamma(\{\max_{x,y\in W_a^+}c^{[\ell(x)-\ell(y)-n]}_{x,y}\})\leq|\Phi|.$$
When $\Phi$ has rank $\leq 2$, the lower bound $r-1$ must be replaced by $1$ (resp., $1$, $2$, $4$) in type $A_1$ (resp.,
$A_2$, $B_2$, $G_2$).\end{thm}

\section{Borel subalgebras}\label{borel}The authors do not know if the bounds on Kazhdan-Lusztig polynomial coefficients
have purely combinatorial arguments, avoiding the representation theory of quantum groups. We present here a hybrid approach,
involving both representation theory and combinatorics. It provides an alternate approach to bounding Kazhdan-Lusztig
coefficients.

Form the $\mathbb N$-graded $\mathbb C$-algebra (basic, but without identity)
\begin{equation}\label{extalgebra}
{\mathbb E}:=\bigoplus_{\lambda,\nu\in\Xregl}\Ext_{U_\zeta}^\bullet(L_\zeta(\lambda),L_\zeta(\nu))
\end{equation}
with multiplication defined by Yoneda product. For $\lambda\in\Xregl$, let $e_\lambda\in\mathbb E$ denote
the idempotent $1_{L_\zeta(\lambda)}
\in\Hom_{U_\zeta}(L_\zeta(\lambda),L_\zeta(\lambda))$. The  set $\Xregl$ is a poset, putting
$\lambda\leq\nu\iff \nu-\lambda\in{\mathbb Z}\Pi$. For any finite ideal (i.~e., saturated subset) $\Gamma$
of $\Xregl$, put $e_\Gamma:=\sum_{\gamma\in\Gamma}e_\gamma\in\mathbb E$, and define
\begin{equation} {\mathbb E}_\Gamma:=e_\Gamma{\mathbb E}e_\Gamma.\end{equation}
Thus, ${\mathbb E}_\Gamma$ is a finite dimensional subalgebra (with identity) of $\mathbb E$. A main
result, proved in \cite[\S5]{PS4}, establishes that ${\mathbb E}_\Gamma$ is a Koszul algebra. Additionally,
${\mathbb E}_\Gamma$ is a quasi-hereditary algebra with weight poset $\Gamma$.\footnote{More precisely,
${\mathbb E}_\Gamma$ has a graded Kazhdan-Lusztig theory in the sense of \cite{CPS1}.}

We will make use of the notion of a Borel subalgebra $B$ of a quasi-hereditary algebra $A$ over an algebraically closed
field $k$; see \cite[\S 2]{PSW}. Let $\Lambda$ be the weight poset of $A$, and, given $\lambda\in\Lambda$, let
 $\Delta_A(\lambda)$ (resp., $\nabla_A(\lambda)$, $L_A(\lambda)$) be the associated standard (resp., costandard, irreducible)
$A$-modules. Assume that $B$ is also quasi-hereditary with the same poset $\Lambda$. The inclusion $B\hookrightarrow A$
induces an exact restriction functor $\Psi:\Amod\to\Bmod$. Then $B$ is a Borel subalgebra of $A$ provided that $\Psi$
admits a left adjoint $\Psi_!:\Amod\to\Bmod$ such that: (i) for any $\lambda\in\Lambda$,  $\Psi\nabla_A(\lambda)$ is isomorphic
to a submodule of $\nabla_B(\lambda)$; (ii) for any $\lambda\in\Lambda$, $\Delta_B(\lambda)\cong L_B(\lambda)$; and (iii)
relative to some poset structure on $\Lambda$, $\Bmod$ is a directed category (in the sense of \cite[p. 218]{PSW}).\footnote{A  more general definition
of Borel subalgebra $B$ is given in \cite[Defn. 2.2]{PSW}, allowing a larger weight poset for $B$. Also, that paper describes a menagerie of different possible
Borel subalgebras.}

Returning to the algebra $\mathbb E$, let $\mathbb B$ (resp., ${\mathbb B}^-$) to be the subalgebra generated by the
idempotents $e_\lambda$, $\lambda\in\Xregl$ together with
the spaces $e_\lambda {\mathbb E}_1e_\nu$ for $\lambda<\nu$ (resp., $\nu<\lambda$). Here ${\mathbb E}_1=\bigoplus_{\lambda,\nu}\Ext^1_{U_\zeta}(L_\zeta(\lambda),L_\zeta(\nu))$. Use of \cite[Thm. 3.5, Cor. 3.7]{PSW}
(and their proofs) establish the following result.

\begin{thm}\label{borelthm}For any finite ideal $\Gamma$ in $\Xregl$, ${\mathbb B}_\Gamma:=e_\Gamma{\mathbb B}e_\Gamma$ and
${\mathbb B}^-_\Gamma:=e_\Gamma{\mathbb B}^-e_\Gamma$ are Borel subalgebras of ${\mathbb E}_\Gamma$. Also,
the algebra ${\mathbb E}_\Gamma$ has a triangular decomposition
$$ {\mathbb E}_\Gamma={\mathbb B}_\Gamma^-\cdot{\mathbb B}_\Gamma.$$
Finally, the (graded) subalgebras ${\mathbb B}_\Gamma$ and ${\mathbb B}_\Gamma^-$ are tightly graded in the sense that
they are (by construction) generated by terms in grades 0 and 1. \end{thm}

Let $\lambda^-\in C_l^-$. For $y<x\in W_{a,l}$, let $\lambda=y\cdot\lambda^-,\nu=x\cdot\lambda^-$, and set $\mu(\lambda,\nu)$
equal to the coefficient of $t^{\ell(x)-\ell(y)-1}$ in the Kazhdan-Luztig polynomial $P_{y,x}$. Also, define
$\mu(\nu,\lambda)=\mu(\lambda,\nu)$.

\begin{cor}\label{borelcorollary}With the above notation, for any non-negative integer $n$,
$$\sum_{\nu}\dim\Ext^n_{U_\zeta}(L_\zeta(\lambda),L_\zeta(\nu))\leq\sum_{(\lambda_1,\cdots\lambda_{n-1},\nu)}
\mu(\lambda,\lambda_1)\mu(\lambda_1,\lambda_2)\cdots\mu(\lambda_{n-1},\nu),$$
where the sum runs over all sequences of  the form $\lambda=\lambda_0>\lambda_1>\cdots>\lambda_s<\lambda_{s+1}<\cdots<\lambda_n=\nu$
for dominant weights $\lambda_1,\cdots,\lambda_{n-1},\lambda_n=\nu$.\end{cor}

\begin{rems}
(a)  This discussion suggests the basic question of determining
$$R_\Phi:=\text{\rm Max}_{x\in W_a^+}\left( \sum_{y\in W_a^+} \mu(x,y)\right).$$
 Determination of
 $$R'_\Phi:=\sum_{y\in W_a^+} \mu(w_0,y)$$
(or a good bound for it) is also an open problem, related to bounding $1$-cohomology (and the Guralnick conjecture \cite{G}).
It is currently open whether $\mu(w_0,y)$ is bounded over all $\Phi$, with $3$ the largest
known value; see \cite{S2}. By \cite{SX},  $\max_{x,y}\mu(x,y)\to\infty$  with larger (affine type $A$) root systems.
In particular, the constant $C(\Phi,l)$ in Theorem \ref{maintheorem} must depend on $\Phi$ and tend to infinity as $\Phi$ gets large.
Conceivably, in the spirit of Guralnick's conjecture, one might replace $C(\Phi,l)$ by a universal constant
if $\lambda$ is fixed as $\lambda=0$ and $\nu$ is allowed to be arbitrary.

(b) It seems likely that the Borel subalgebras in Theorem \ref{borelthm} are
themselves Koszul (or at least quadratic). Assuming this is the case, it also seems likely that the quadratic relations for ${\mathbb E}_\Gamma$
can be taken to be of three types: (i) the quadratic relations of ${\mathbb B}_\Gamma$; (ii) the quadratic relations
of ${\mathbb B}^-_\Gamma$; and (iii) the relations expressing $b\cdot b^-$ ($b\in{\mathbb B}_\Gamma$, $b^-\in
{\mathbb B}_\Gamma^-$) as a linear combination of terms
$c^-\cdot c$ ($c\in{\mathbb B}_\Gamma$, $c^-\in{\mathbb B}_\Gamma^-$).

Understanding a presentation for the algebras ${\mathbb B}_\Gamma$, ${\mathbb B}^-_\Gamma$ could
dramatically improve the estimates in Corollary \ref{borelcorollary}. Essentially, the corollary as written only
estimates the degree $m$ of the dimensions of the graded subspaces ${\mathbb B}_{\Gamma,n-s}$ and ${\mathbb B}^-_{\Gamma,s}$
as degree $n-s,s$-fold tensor products of ${\mathbb B}_{\Gamma,1}$ and ${\mathbb B}^-_{\Gamma,1}$ over ${\mathbb B}_{\Gamma,0}$.
No relations for the quiver algebra ${\mathbb B}_\Gamma$ and ${\mathbb B}^-_\Gamma$ have yet been considered.
\end{rems}

\section{Complexity for rational $G$-modules}\label{stable}
In this section, $G$ is a simple, simply connected
algebraic group over an algebraically closed field of characteristic $p>0$.
For a positive integer $e$, let $X_{e,p}$ be the set of $p^e$-restricted dominant
weights $\lambda$ (i.~e., $(\lambda,\alpha^\vee)<p^e$ for all $\alpha\in\Pi$). Assume
that $G$ is defined over ${\mathbb F}_p$ and let $F:G\to G$ be the Frobenius morphism.
For a rational $G$-module $V$, $V^{(r)}$ denotes the rational $G$-module obtained from $V$ by
making $g$ act on $V$ as $F^r(g)$.

If $V$ is a finite dimensional rational
$G$-module (over $k$ of characteristic $p$) and $m$ is a non-negative integer, define
$$
\gamma_m(V)=\max_{L\,\text{\rm irred}}\dim\Ext^m_G(V,L).$$

Also, put
$$\begin{cases}
\gamma_m(\Phi,e,p)=\max_{\lambda\in X_{e,p}}\gamma_m(L(\lambda));\\
\gamma_m(\Phi,e)=\max_p\gamma_m(\Phi,e,p).\end{cases}
$$
By \cite[Thm. 7.1]{PS5}, these maximums are finite.  However, we do not know ($V\not=0$) if any of these sequences $\{\gamma_m(V)\}$,
$\{\gamma_m(\Phi,e,p)\}$, and $\{\gamma_m(\Phi,e)\}$ has polynomial growth.  We can, however, prove the following result
establishing a kind of exponential growth rate bound on $\{\gamma_m(\Phi,e,p)\}$ and $\{\gamma_m(\Phi,e)\}$.

\begin{thm}\label{log} (a) The sequence $\{\log\gamma_m(\Phi,e)\}$ has polynomial rate of growth at most 4.

(b) For any fixed prime $p$, the sequence $\{\log\gamma(\Phi,e,p)\}$ has polynomial rate of growth at most 3.
\end{thm}

\begin{proof} The proof involves careful use of the arguments given in \cite[\S7]{PS5}. Fix the root system $\Phi$ (which, we assume, as in most of this paper, is indecomposable). The
constants below will generally depend on $\Phi$.

First, there is a positive integer $b=b(\Phi)$ such that $b\geq 2h-2$ and
such that the LCF holds for any simple, simply connected algebraic group $G$ over an algebraically closed field $p$ provided
that $p\geq b$ \cite{Fiebig}. Let $e$ be a non-negative integer. It is shown in \cite[Thm. 7.1]{PS5} that there exists a positive
constant $c(m,e)$ such that if $\lambda,\mu\in X^+$ with $\lambda$ being $p^e$-restricted then
$\dim\Ext^m_G(L(\lambda),L(\mu))\leq c(m,e)$, as long as $p\geq b$. We can obviously take $c(0,e)=1$ for all $e$. Suppose
that $m>0$ is fixed and that $c(i,e)$ are defined for $0\leq i<m$ and all $e$. Below we will provide
a formula for $c(m,e)$ that bounds all the $\Ext^m_G(L(\lambda),L(\mu))$ (with $\lambda$ $p^e$-restricted). It will
be slightly modified from a similar bound in \cite{PS5} to a bound more suitable for our purposes.

By \cite[Lem. 6.3]{PS5} that there exists a constant $C=C(\Phi)$ which bounds the length of any PIM $Q_\zeta(\lambda)$
for a quantum enveloping algebra $U_\zeta$ of type $\Phi$ for $l$ odd, not divisible by $3$ in type $G_2$, and otherwise
arbitrary. Let $M$ be the maximum of the dimensions of all irreducible modules $L_{\mathbb C}(\tau)$ for the complex
simple Lie algebra $\mathfrak g$ (of type $\Phi$) for dominant weights $\tau=\sum a_i\varpi_i$ with $a_i\leq 2h-2$. It is
proved in \cite[Lem. 7.2]{PS5} that if $p\geq b$, and if $\nu$ is $p^r$-restricted, then
$$\text{\rm length}\,Q_r(\nu)\leq C^{(r+4)(r-1)+1}M^{r-1}.$$
Here $Q_r(\nu)$ is the PIM in $G_r$-mod for $L(\nu)$ (which remains irreducible upon restriction to $G_r$). We can
replace $C$ by a larger constant to assume that
$$
  {\text{\rm length}}\,Q_r(\nu)\leq C^{r^2}.$$

Define
$f_p(m)= \left[\frac{ctm}{p-1}\right]+1$, where $c$ is the maximum coefficient of the maximal root of $\Phi$ (when expressed
 as a linear combination of simple roots), $t$ is the torsion exponent of $X/{\mathbb Z}\Phi$, and $[\,\,\,]$ is the
largest integer function.
Since $p\geq 2h-2$, it follows that $f_p(m)\leq (t+1)m$ for all $p$. Replacing $f(\Phi,m)$ by $(t+1)m$
 in the proof of \cite[Thm. 7.1]{PS5}, we see that
$$\dim\Ext^m_G(L(\lambda),L(\nu))\leq C^{e^2}c(m-1,e+1)+C^{(m+1)^2}c(m-1,(t+1)m+1).$$
Then we define
$$
c(m,e)= C^{e^2}c(m-1,e+1)+C^{(m+1)^2}c(m-1,(t+1)m+1).$$
We can expand this expression into a sum of $2^m$ terms, each consisting of a product of $m$ terms   $C^{u^2}$, where $u$ is of the form $e'$ or $m'+1$, with $0\leq e'\leq \max\{e,t+1\}+m$
and $0\leq m'\leq m$.
It is clear this sum is bounded by an expression of the form $D^{m^3}$.

So far, we have shown that, if $p$ is sufficiently large (depending on $\Phi$), then there is a number $D$ such that
for all such large $p$, $\gamma_m(\Phi,e,p)\leq D^{m^3}$.
Thus, (a) will follow once (b) is proved.

We now consider an individual fixed prime $p$.  The overall structure of the argument for (b) is similar to the
discussion above, using constants $c_p(m,e)$ in place of $c(m,e)$, used above.
In place of the modules $Q_r(\lambda)$,
we use  $G$-modules $\widehat Q_r(\lambda)$ which have the property that
\begin{equation}\label{property}\begin{cases}\widehat Q_r(\lambda)|_{G_r}\cong Q_r(\lambda)^{\oplus n};\\
\text{\rm head}_G\widehat Q_r(\lambda)\supseteq L(\lambda)\end{cases}
\end{equation}
for some $n=n(\lambda)$. The existence of such $G$-modules is guaranteed by \cite[Cor. 8.5]{PS5}, for all
$r>0$. However, we need here some control on the growth of their
dimensions as $r$ gets larger. Thus, we  note that $G$-modules $\widehat Q_r(\lambda)$ satisfying (\ref{property}) may be constructed from those for $r=1$ using Frobenius
twists and tensor products. (See (\ref{property2} below.) With this construction, $\dim \widehat Q_r(\lambda)\leq C_1^r$, where $C_1=C_1(p)=\max_{\mu\in X_{1,p}}
\widehat Q_1(\mu)$ which enter. But, of course, length$(\widehat Q_r(\lambda))\leq \dim \widehat Q_r(\lambda)$. Also,
$$\text{\rm length}(L(-w_0\lambda)\otimes L(\lambda))\leq\dim L(-w_0\lambda)\otimes L(\lambda)\leq p^{|\Phi|r}.$$
Let $\theta$ be a positive integer such
that all composition factors of
$\widehat Q_1(\mu)$ are $p^{\theta}$-restricted, for all $\mu\in X_{1,p}$.
Thus, for $e'=e + \theta + 1 +
[\log_p\frac{h-1}{p-1}]$, all composition factors
of
\begin{equation}\label{property2}\widehat Q_e(\lambda)=\widehat Q_1(\lambda_0)\otimes \widehat Q_1(\lambda_1)^{(1)}\otimes\cdots\otimes \widehat Q_1(\lambda_{e-1})^{(e-1)}\end{equation}
 and of $L(-w_0\lambda)\otimes L(\lambda)$ are $p^{e'}$-
restricted, for all  $p^e$-restricted $\lambda=\lambda_0+p\lambda_1+\cdots +p^{e-1}\lambda_{e-1}$ ($\lambda_i\in X_{1,p}$).
(Compute an inner product of any dominant weight of $\widehat Q_e(\lambda)$ or of $L(-w_0\lambda)\otimes L(\lambda)$ with $\alpha_0^\vee$.)
The (small $p$ part of the) proof of \cite[Thm. 7.1]{PS5} now gives, for $e\geq\left[\frac{ctm}{p-1}\right]+1$,
$$\dim\Ext^m_G(L(\lambda),L(\nu))\leq C_1^{e'}p^{|\Phi|e}c_p(m-1,e^{\prime\prime}).$$
Here  $e^{\prime\prime}=(e^{\prime })^\prime$.
The right-hand side of this inequality is clearly bounded above by $D_1^ec_p(m-1,e+a)$ for some integer
$D_1=D_1(p)\geq 0$ and $a=a(p)=2(\theta +1 +[\log_p\frac{h-1}{p-1}])$. Setting
$c_p(m,e)=D^e_1c_p(m-1,e+2),$ we can take
$$c_p(m,r)=D_1^eD_1^{e+a}\cdots D_1^{e+m^2}\leq D(p)^{m^2},$$
for a suitable constant $D(p)=D(p,e)$.

 This completes the proof part (b), and hence of (a). \end{proof}

 We have proved there exist constants $D_m(\Phi)$, $m\geq 0$, depending only
on $\Phi$, such that for any $G$ with root system $\Phi$,
$\dim\opH^m(G,L)\leq D(\Phi)^{m^3}$. As an immediate consequence, we have
$$\dim\opH^m_{\text{\rm gen}}(G,L)\leq D_m(\Phi)^{m^3},$$
where the generic cohomology $\opH^\bullet_{\text{\rm gen}}(G,V)$ is defined, in degree $m$, as the stable limit
$$\opH^m(G,V)=\lim_\sigma \opH^m(G_\sigma,V)=\lim_{r\to\infty}H^m(G,V^{(r)})$$
of cohomology spaces for underlying finite groups $G_\sigma$ of Lie type;
see \cite{CPSK}, \cite{Avr}. In most cases, $G_\sigma$ is a central extensions of a finite simple
group, and all finite simple groups of Lie type arise in this way.
In \cite[Ques. 12.1]{GKKL}, it is asked, for what $m$ is there an absolute bound $E_m$, depending only
on $m$, such that
$$\dim\opH^m(H,L)\leq E_m$$
for all finite simple groups $H$ and all irreducible $kH$-modules $L$, where $k$ is an algebraically closed field
of arbitrary characteristic. If we suppose that  \cite[Ques. 12.1]{GKKL} has a positive answer, for all $m$, the results in this section
would suggest additional structure for the constants $E_m$.  In particular, we can could ask if there is
a constant $E$ with
$$E_m\leq E^{m^3},\,\,\forall \,\,m.$$
Similarly, we can ask if there is a constant $D$ with
$\dim \opH^m(G,L)\leq D^{m^3}$
in the case of algebraic groups $G$; that is, can $D(\Phi)$ be taken independent of $\Phi$?  Even if this is not the case,
it would be interesting to investigate the growth of $D(\Phi)$ with respect to $\Phi$.

Should it turn out that $\gamma_m(\Phi,e)$ has polynomial growth, even for $e=0$, the above questions could all be sharpened.
  Could there exist
constants $E$ and $f$ with $E_m\leq Em^f$? Or even such constants $E=
E(\Phi)$ for each $\Phi$? In the algebraic group case, we pose the following formal question.

\begin{ques}\label{Gcase} Let $\Phi$ be a (finite) root system. Do there exist constants $C=C(\Phi)$ and $f=f(\Phi)$ such that
$$\dim {\opH}^m(G,L)\leq Cm^f$$
for all semisimple, simply connected groups $G$ over an algebraically closed field $k$ (of arbitrary characteristic) having
root system $\Phi$, and
all irreducible rational $G$-modules $L$?
\end{ques}

\begin{rem}\label{completeness}
A positive answer to Question \ref{Gcase} would imply that the sequence $\{\gamma_m(V)\}$ has polynomial growth when $V=L(0)$,
and it is interesting to ask for any finite dimensional rational G-module $V$,
if this sequence has polynomial growth. Set $\cx(V)=\infty$ when this is
not the case, and otherwise let $\cx(V)$ denote the growth rate of ${\gamma_m(V)}$.
Conceivably, it could be true that $\cx(V)=\cx(V')$,
for all irreducible rational $G$-modules $V,V'$.   Suppose it could be shown that $\cx(V\otimes M)
\leq\cx(V)$, for all finite dimensional rational $G$-modules $M$. (We do not know if this is the case, but raise the validity
 of the inequality as a question.
The corresponding inequality does hold for finite groups.) Then it can be shown that $\cx(V)=\cx(V')$ for all irreducible
$G$-modules $V,V'$. We omit the proof, but note that it depends on the ($\Ext$-group preserving) full embedding $M\mapsto L((p^r-1)\rho)\otimes M^{(r)}$
of the category of finite dimensional rational $G$-modules into itself.

The equality of the $\cx(V)$'s for irreducible modules $V$, if it holds,
implies that either no nonzero finite dimensional rational $G$-module $V$
has finite $\cx(V)$, or all do. This would be a kind of ``completeness"
property, to borrow a term from computer science complexity theory.\footnote{The class of NP complete problems in computer science complexity theory has
the property that, if any one of them (can be solved by an algorithm that) runs in a polynomial number of time steps,
then all problems in the class NP run in polynomial time. In particular, the
question of whether any one NP complete problem runs in polynomial time is
completely equivalent to the same question for any other NP complete problem. }
\end{rem}

Finally, in analogy with \S5 above, define (for $G$ as above with $p\geq h$)
$${\mathbb E}_p=\bigoplus_{\lambda,\nu\in\Xreg}\Ext^\bullet_G(L(\lambda),L(\nu)).$$
For any finite ideal $\Gamma$ of $p$-regular dominant weights, we can form the algebra ${\mathbb E}_{p,\Gamma}=e_\Gamma{\mathbb E_p}e_{\Gamma}$ as above. Little seems to be known concerning the structure of these finite dimensional algebras, unless
the ideal $\Gamma$ is confined to the Jantzen region and the Lusztig character formula holds. (In those cases, ${\mathbb E}_{p,\Gamma}$
is a quasi-hereditary Koszul algebra \cite[\S8]{PS4}.)  However, the following question seems reasonable. We do not know the
answer even in type $A_1$.

\begin{ques} For any ideal $\Gamma$, is the ${\mathbb N}$-graded algebra ${\mathbb E}_{p,\Gamma}$ generated in positive degrees by its grade 1 terms?
\end{ques}

In addition, we can ask if ${\mathbb E}_{p,\Gamma}$ has a triangular decomposition ${\mathbb E}_{p,\Gamma}={\mathbb B}_{p,\Gamma}^-
{\mathbb B}_{p,\Gamma}$, where ${\mathbb B}_{p,\Gamma}$
and ${\mathbb B}_{p,\Gamma}^-$ are Borel subalgebras of ${\mathbb E}_{p,\Gamma}$ defined as in Theorem \ref{borelthm}, replacing
$U_\zeta$ by $G$. Again, these statements hold in the Jantzen region, with each of ${\mathbb B}^-_{p,\Gamma}$ and
${\mathbb B}_{p,\Gamma}$ generated (by construction) in positive degree by their grade 1 terms.

\end{document}